\documentclass{amsart}
\usepackage[all]{xy}
\usepackage{color}
\usepackage{amsthm}
\usepackage{amssymb}
\usepackage[colorlinks=true]{hyperref}

\numberwithin{equation}{section}

\newtheorem{theorem}[equation]{Theorem}
\newtheorem*{theorem*}{Theorem}

\newtheorem{proposition}[equation]{Proposition}

\newtheorem*{corollary*}{Corollary}

\theoremstyle{remark}
\newtheorem{definition}[equation]{Definition}

\newtheorem{example}[equation]{Example}

\theoremstyle{remark}
\newtheorem{remark}[equation]{Remark}

\newcommand{\cA}{{\mathcal A}}
\newcommand{\cB}{{\mathcal B}}
\newcommand{\cC}{{\mathcal C}}
\newcommand{\cD}{{\mathcal D}}

\newcommand{\cM}{{\mathcal M}}

\newcommand{\cO}{{\mathcal O}}

\newcommand{\cU}{{\mathcal U}}

\newcommand{\bbL}{\mathbb{L}}

\newcommand{\bbZ}{\mathbb{Z}}

\DeclareMathOperator{\Mot}{Mot}
\DeclareMathOperator{\NChow}{NChow} 

\DeclareMathOperator{\Mix}{KMM} 
\DeclareMathOperator{\KPM}{KPM} 
\DeclareMathOperator{\KTM}{KTM} 

\newcommand{\Sp}{\mathsf{Spt}} 

\newcommand{\dgcat}{\mathsf{dgcat}}

\newcommand{\dg}{\mathsf{dg}}

\newcommand{\Hom}{\mathrm{Hom}}

\newcommand{\rep}{\mathrm{rep}}

\newcommand{\Ho}{\mathsf{Ho}}
\newcommand{\HO}{\mathsf{HO}}

\newcommand{\Hmo}{\mathsf{Hmo}}
\newcommand{\op}{\mathsf{op}}

\newcommand{\too}{\longrightarrow}

\newcommand{\ie}{\textsl{i.e.}\ }
\newcommand{\eg}{\textsl{e.g.}}
\newcommand{\etc}{\textsl{etc.}}

\newcommand{\Madd}{\Mot^{\mathsf{add}}_\dg} 

\newcommand{\Mloc}{\Mot^{\mathsf{loc}}_\dg} 

\begin{document}

\title[Weight structure on noncommutative motives]{Weight structure on noncommutative motives}

\author{Gon{\c c}alo~Tabuada}

\address{Gon{\c c}alo Tabuada, Department of Mathematics, MIT, Cambridge, MA 02139, USA}
\email{tabuada@math.mit.edu}
\urladdr{http://math.mit.edu/~tabuada}


\date{\today}

\thanks{The author was partially supported by the J.H. and E.V. Wade award. He is very grateful to Paul Balmer, Yuri Manin and Greg Stevenson for useful conversations.}


\abstract{In this note we endow Kontsevich's category $\Mix_k$ of noncommutative mixed motives with a non-degenerate weight structure in the sense of Bondarko. As an application we obtain a convergent weight spectral sequence for every additive invariant (\eg~algebraic $K$-theory, cyclic homology, topological Hochschild homology, \etc), and a ring isomorphism between $K_0(\Mix_k)$ and the Grothendieck ring of the category of noncommutative Chow motives.
}}
\maketitle
\vskip-\baselineskip
\vskip-\baselineskip
\section{Weight structure}
In his seminal talk~\cite{IAS}, Kontsevich introduced the triangulated category $\Mix_k$ of noncommutative mixed motives (over a base commutative ring $k$) and conjectured the existence of a ``different'' $t$-structure on this category. In this note we formalize Kontsevich's beautiful insight and illustrate some of its important consequences. Recall from \cite{IMRN,CvsNC} the construction of the additive category $\NChow_k$ of noncommutative Chow motives. Our formalization of the ``different'' $t$-structure is the following:
\begin{theorem}\label{thm:main}
There exist two full subcategories $\Mix_k^{w\geq 0}$ and $\Mix_k^{w\leq 0}$ of $\Mix_k$ verifying the following seven conditions:
\begin{itemize}
\item[(i)] $\Mix_k^{w\geq 0}$ and $\Mix_k^{w\leq 0}$ are additive and idempotent complete;
\item[(ii)] $\Mix_k^{w\geq 0} \subset \Mix_k^{w\geq 0}[1]$ and $\Mix_k^{w\leq 0}[1] \subset \Mix_k^{w\leq 0}$;
\item[(iii)] For every $M \in \Mix_k^{w \geq 0}$ and $N \in \Mix_k^{w \leq0}[1]$ we have $\Hom_{\Mix_k}(M,N)=~0$;
\item[(iv)] For every $M \in \Mix_k$ there is a distinguished triangle
$$ N_2[-1] \too M \too N_1 \too N_2$$
with $N_1 \in \Mix_k^{w \leq 0}$ and $N_2 \in \Mix_k^{w \geq 0}$;
\item[(v)] $\Mix_k = \cup_{l \in \bbZ} \Mix_k^{w\geq 0}[-l] = \cup_{l \in \bbZ} \Mix_k^{w\leq 0}[-l]$; 
\item[(vi)] There is a natural equivalence of categories $\NChow_k \simeq \Mix_k^{w \geq 0} \cap \Mix_k^{w \leq 0}$;
\item[(vii)] $\cap_{l \in \bbZ} \Mix_k^{w \geq 0}[-l] = \cap_{l\in \bbZ} \Mix_k^{w \leq 0} [-l] =\{0\}$.
\end{itemize} 
\end{theorem}
Items (i)-(iv) assert that the triangulated category $\Mix_k$ is endowed with a weight structure $w$ (also known in the literature as a co-$t$-structure) in the sense of Bondarko~\cite[Def.~1.1.1]{Bondarko}. Item (v) asserts that $w$ is bounded, item (vi) that the heart of $w$ can be identified with the category of noncommutative Chow motives, and item (vii) that $w$ is non-degenerate. Theorem~\ref{thm:main} should then be regarded as the noncommutative analogue of the Chow weight structure on Voevodsky's triangulated category of motives; consult \cite[\S6.5-6.6]{Bondarko}.
\section{Weight spectral sequences}
Let $\dgcat_k$ be the category of (small) dg categories over a fixed base commutative ring $k$; consult Keller's ICM address \cite{ICM}.
\begin{definition}\label{def:additive}
Let $L(-):\dgcat_k \to \cM$  be a functor with values in a symmetric monoidal stable model category; see \cite[\S4 and \S7]{Hovey}. We say that $L$ is an {\em additive invariant} if it verifies the following three conditions:
\begin{itemize}
\item[(i)] filtered colimits are mapped to filtered colimits;
\item[(ii)] {\em derived Morita equivalences} (\ie dg functors which induce an equivalence on the associated derived categories; see \cite[\S4.6]{ICM}) are mapped to weak equivalences;
\item[(iii)] {\em split exact sequences} (\ie sequences of dg categories which become split exact after passage to the associated derived categories; see \cite[\S13]{Duke}) are mapped to direct sums
\begin{eqnarray*}
\xymatrix@C=1.5em@R=1.0em{
0 \ar[r] &  \cA \ar[r]  & \cB \ar[r]  \ar@/_0.5pc/[l] & \cC \ar[r] \ar@/_0.5pc/[l] &  0
} &\mapsto&
L(\cA) \oplus L(\cC) \simeq L(\cB)
\end{eqnarray*}
in the homotopy category $\Ho(\cM)$.
\end{itemize}
\end{definition}
\begin{proposition}\label{prop:aux}
Every additive invariant $L(-)$ gives rise to a triangulated functor $L(-): \Mix_k \too \Ho(\cM)$; which we still denote by $L(-)$.
\end{proposition}
Consider the following compositions
\begin{equation}\label{eq:induced}
\xymatrix{
L_n(-): \Mix_k \ar[r]^-{L(-)} & \Ho(\cM)  \ar[rr]^-{\Hom({\bf 1}[n],-)} && \mathsf{Ab} \qquad n \in \bbZ\,,
}
\end{equation}
where ${\bf 1}$ stands for the $\otimes$-unit of $\cM$ and $\mathsf{Ab}$ for the category of abelian groups.
\begin{example}[Algebraic $K$-theory]
Recall from \cite[\S5.2]{ICM} that the (connective) algebraic $K$-theory functor $K(-): \dgcat_k \to \Sp$, with values in the category of spectra, satisfies the above conditions (i)-(iii) and hence is an additive invariant.
\end{example}

\begin{example}[Hochschild and cyclic homology]
Recall from \cite[\S5.3]{ICM} that the Hochschild and cyclic homology functors $HH(-), HC(-): \dgcat_k \to \cC(k)$, with values in the category of complexes of $k$-modules, are additive invariants. In these examples the associated functors $HH_n(-)$ and $HC_n(-)$ take values in the abelian category $k\text{-}\mathrm{Mod}$ of $k$-modules. 
\end{example}

\begin{example}[Negative cyclic homology]
Recall from \cite[Example~7.10]{CT1} that the mixed complex functor $C(-): \dgcat_k \to \cC(\Lambda)$, with values in the category of mixed complexes, satisfies the above conditions (i)-(iii) and hence is an additive invariant. Moreover, as explained in \cite[Example~8.10]{CT1}, the associated functors $C_n(-)$ agree with the negative cyclic homology functors $HC^{-}_n(-)$.
\end{example}

\begin{example}[Periodic cyclic homology]\label{ex:HP}
Periodic cyclic homology is {\em not} an additive invariant since its definition uses infinite products and these do not commute with filtered colimits. Nevertheless, it factors through $\Mix_k$ as follows: recall from \cite[Example~7.11]{CT1} that we have a $2$-perioditization functor $P(-): \cC(\Lambda) \to k[u]\text{-}\mathrm{Comod}$, with values in the category of comodules over the Hopf algebra $k[u]$. This functor preserves weak equivalences and hence by applying the above Proposition~\ref{prop:aux} to $C(-)$ we obtain the following composed triangulated functor
$$ \Mix_k \stackrel{C(-)}{\too} \Ho(\cC(\Lambda)) \stackrel{P(-)}{\too} \Ho(k[u]\text{-}\mathrm{Comod})\,.$$
As explained in \cite[Example~8.11]{CT1}, the associated functors $((P\circ C)(-))_n$ agree with the periodic cyclic homology functors $HP_n(-)$. 
\end{example}

\begin{example}[Topological Hochschild homology]
Recall from \cite{BM} (see also \cite[\S8]{AGT}) that the topological Hochschild homology functor $THH(-): \dgcat_k \to \Sp$ is also an example of an additive invariant.
\end{example}
Recall from \cite[Thm.~2.8]{Prods} and \cite[Prop.~2.5]{Criterion} the construction of the following natural transformations between additive invariants:
\begin{equation}\label{eq:Chern}
tr:K(-) \Rightarrow HH(-)  \quad ch^{2i}:K(-) \Rightarrow HC(-)[-2i] \quad ch^-:K(-) \Rightarrow C(-)\,.
\end{equation}
By first evaluating these natural transformations at a noncommutative mixed motive $M$, and then passing to the associated functors \eqref{eq:induced} we obtain, respectively, the Dennis trace maps, the higher Chern characters, and the negative Chern characters:
\begin{equation*}
tr_n: K_n(M) \to HH_n(M) \quad ch^{2i}_n:K_n(M) \to HC_{n+2i}(M) \quad ch^-_n: K_n(M) \to HC^-_n(M)\,.
\end{equation*}
\begin{theorem}\label{thm:main2}
Under the preceding notations the following holds:
\item[(i)] To every noncommutative mixed motive $M$ we can associate a cochain (weight) complex of noncommutative Chow motives
\begin{eqnarray*}
t(M): && \cdots \too M^{(i-1)} \too M^{(i)} \too M^{(i+1)} \too \cdots\,.
\end{eqnarray*}
Moreover, the assignment $M \mapsto t(M)$ gives rise to a conservative functor from $\Mix_k$ towards a certain weak category of complexes $K_{\mathfrak{m}}(\NChow_k)$; consult \cite[\S3.1]{Bondarko}.
\item[(ii)] Every additive invariant $L(-)$ yields a convergent (weight) spectral sequence
\begin{equation}\label{eq:spectral-eq}
E_1^{pq}(M) = L_{-q}(M^{(p)}) \Rightarrow L_{-p -q}(M)\,.
\end{equation}
Moreover, \eqref{eq:spectral-eq} is functorial on $M$ after the $E_1$-term.
\item[(iii)] The above natural transformations \eqref{eq:Chern} respect the spectral sequence \eqref{eq:spectral-eq}.
\end{theorem}
Intuitively speaking, item (i) of Theorem~\ref{thm:main2} shows us that all the information concerning a noncommutative mixed motive can be encoded into a cochain complex. Items (ii) and (iii) endow the realm of noncommutative motives with a new powerful computational tool which is moreover well-behaved with respect to the classical Chern characters. We intend to develop this computational aspect in future work.
\section{Grothendieck rings}
As explained in \cite{IAS,IMRN,CvsNC}, the categories $\Mix_k$ and $\NChow_k$ are endowed with a symmetric monoidal structure induced by the tensor product of dg categories. Hence, the Grothendieck group of $\Mix_k$ (considered as a triangulated category) and the Grothendick group of $\NChow_k$ (considered as an additive category) are endowed with a ring structure.
\begin{theorem}\label{thm:main3}
The equivalence of categories of item $\mathrm{(vi)}$ of Theorem~\ref{thm:main} gives rise to a ring isomorphism
$$K_0(\NChow_k) \stackrel{\sim}{\too} K_0(\Mix_k)\,.$$
\end{theorem}
Informally speaking, Theorem~\ref{thm:main3} shows us that ``up to extension'' the categories $\Mix_k$ and $\NChow_k$ have the same isomorphism classes.
\section{Proofs}\label{sec:proofs}

\subsection*{Proof of Theorem~\ref{thm:main}}
Recall from \cite{IAS} that a dg category $\cA$ is called {\em smooth} if it is perfect as a bimodule over itself and {\em proper} if for each ordered pair of objects $(x,y)$ in $\cA$, the complex of $k$-modules $\cA(x,y)$ is perfect. Recall also that Kontsevich's construction of $\Mix_k$ decomposes in three steps:
\begin{itemize}
\item[(1)] First, consider the category $\KPM_k$ (enriched over spectra) whose objects are the smooth and proper dg categories, whose morphisms from $\cA$ to $\cB$ are given by the (connective) algebraic $K$-theory spectrum $K(\cA^\op \otimes^\bbL \cB)$, and whose composition is induced by the (derived) tensor product of bimodules. 
\item[(2)] Then, take the formal triangulated envelope of $\KPM_k$. Objects in this new category are formal finite extensions of formal shifts of objects in $\KPM_k$. Let $\KTM_k$ be the associated homotopy category.
\item[(3)] Finally, pass to the pseudo-abelian envelope of $\KTM_k$. The resulting category $\Mix_k$ is what Kontsevich named the category of noncommutative mixed motives.
\end{itemize}
Recall from \cite[\S15]{Duke} the construction of the additive motivator of dg categories $\Madd$ and the associated base triangulated category $\Madd(e)$. The analogue of \cite[Prop.~8.5]{CT1} (with the above definition\footnote{In \cite[\S8.2]{CT1} we have considered non-connective algebraic $K$-theory since we were interested in the relation with To{\"e}n's secondary $K$-theory. However, Kontsevich's original definition is in terms of {\em connective} algebraic $K$-theory.} of $\Mix_k$ and $\Mloc(e)$ replaced by $\Madd(e)$) holds similarly. Hence, $\Mix_k$ can be identified with the smallest thick triangulated subcategory of $\Madd(e)$ spanned by the noncommutative motives of smooth and proper dg categories. Similarly, $\KTM_k$ can be identified with the smallest triangulated subcategory of $\Madd(e)$ spanned by the noncommutative motives of smooth and proper dg categories. In what follows, we will assume that these identifications have been made.

Now, recall from \cite{IMRN,CvsNC} that the category $\NChow_k$ of noncommutative Chow motives is defined as the pseudo-abelian envelope of the category whose objects are the smooth and proper dg categories, whose morphisms from $\cA$ to $\cB$ are given by the Grothendieck group $K_0(\cA^\op \otimes^\bbL \cB)$, and whose composition is induced by the (derived) tensor product of bimodules.

\begin{proposition}
There is a natural fully-faithful functor
\begin{equation}\label{eq:fully-faithful}
\Phi: \NChow_k \too \Mix_k\,.
\end{equation}
\end{proposition}
\begin{proof}
Given dg categories $\cA$ and $\cB$, let $\cD(\cA^\op \otimes^\bbL \cB)$ be derived category of $\cA\text{-}\cB$-bimodules and $\rep(\cA,\cB)\subset \cD(\cA^\op \otimes^\bbL \cB)$ the full triangulated subcategory spanned by those $\cA\text{-}\cB$-bimodules $X$ such that for every object $x \in \cA$, the associated $\cB$-module $X(x,-)$ is perfect; consult \cite[\S4.2]{ICM} for further details. Recall from \cite[\S5]{IMRN} the construction of the additive category $\Hmo_0$: the objects are the dg categories, the morphisms from $\cA$ to $\cB$ are given by the Grothendieck group $K_0\rep(\cA,\cB)$ of the triangulated category $\rep(\cA,\cB)$, and the composition is induced by the (derived) tensor product of bimodules. There is a natural functor
\begin{equation}\label{eq:nat-functor}
 \cU_{\mathsf{A}}: \dgcat_k \too \Hmo_0
 \end{equation}
that is the identity on objects and which sends a dg functor $F: \cA\to \cB$ to the class of the corresponding $\cA\text{-}\cB$-bimodule. On the other hand, recall from \cite[\S15]{Duke} the construction of the functor
\begin{equation}\label{eq:functor}
\cU_a: \dgcat_k \too \Madd(e)\,.
\end{equation}
As proved in \cite[Thms.~4.6 and 6.3]{IMRN}, $\cU_{\mathsf{A}}$ is the universal functor with values in an additive category which inverts derived Morita equivalences and sends split exact sequences\footnote{This condition can equivalently be formulated in terms of a general semi-orthogonal decomposition in the sense of Bondal-Orlov; see \cite[Thm.~6.3(4)]{IMRN}.} to direct sums (see condition (iii) of Definition~\ref{def:additive}). Since these conditions are satisfied by the functor $\cU_a$ (see \cite[Thm.~15.4]{Duke}) and $\Madd(e)$ is an additive category (since it is triangulated) we obtain an induced additive functor $\Psi$ making the following diagram commute
$$
\xymatrix{
\dgcat_k \ar[d]_-{\cU_{\mathsf{A}}} \ar[dr]^-{\cU_a} & \\
\Hmo_0 \ar[r]_-\Psi & \Madd(e)\,.
}
$$ 
Let us denote by $\Hmo_0^{\mathsf{sp}} \subset \Hmo_0$ the full subcategory of smooth and proper dg categories. When $\cA$ (and $\cB$) is smooth and proper we have a natural isomorphism
$$ \Hom_{\Hmo_0^{\mathsf{sp}}}(\cA,\cB) := K_0\rep(\cA,\cB) \simeq K_0(\cA^\op \otimes^\bbL \cB)\,;$$
see \cite[Lemma~4.9]{CT1}. Hence, we observe that the category $\NChow_k$ of noncommutative Chow motives is the pseudo-abelian envelope of $\Hmo_0^{\mathsf{sp}}$. Since by construction the triangulated category $\Mix_k\subset \Madd(e)$ is idempotent complete, the composition $\Hmo_0^{\mathsf{sp}}\subset \Hmo_0 \stackrel{\Psi}{\too} \Madd(e)$ extends then to a well-defined additive functor
$$ \Phi: \NChow_k \too \Mix_k\subset \Madd(e)\,.$$
Finally, the fact that $\Phi$ is fully-faithful follows from the following computation
$$\Hom_{\Madd(e)}(\cU_a(\cA), \cU_a(\cB))\simeq K_0\rep(\cA,\cB)\simeq K_0(\cA^\op \otimes^\bbL \cB)$$
for every smooth and proper dg category $\cA$; see \cite[Prop.~16.1]{Duke}.
\end{proof}

Let us now verify the conditions of Bondarko's \cite[Thm.~4.3.2 II]{Bondarko} with $\underline{C}$ the triangulated category $\KTM_k\subset \Mix_k$ and $H$ the essential image of the composition $\Hmo_0^{\mathsf{sp}} \subset \NChow_k \to \Mix_k$. By construction, $H$ generates $\KTM_k$ in the sense of \cite[page~11]{Bondarko}. Moreover, given any two smooth and proper dg categories $\cA$ and $\cB$, we have the following computation:
\begin{equation}\label{eq:computation}
\Hom_{\Mix_k}(\Phi(\cA),\Phi(\cB)[-n]) \simeq \left\{ \begin{array}{ll} K_n(\cA^\op \otimes^\bbL \cB) &  n\geq0 \\  0 & n <0 \,. \end{array} \right.
\end{equation}
This follows from \cite[Prop.~16.1]{Duke} combined with the specific construction of $\Phi$. Hence, $H \subset \KTM_k$ is negative in the sense of \cite[Def.~4.3.1(1)]{Bondarko}. The conditions of \cite[Thm.~4.3.2 II]{Bondarko} are then satisfied and so we conclude that there exists a unique bounded weight structure $w$ on $\KTM_k$ whose heart is the pseudo-abelian envelope of $H$. Note that the heart is then equivalent to $\NChow_k$ under the above fully-faithful functor \eqref{eq:fully-faithful}. Since the weight structure $w$ is bounded, \cite[Prop.~5.2.2]{Bondarko} implies that $w$ can be extended from $\KTM_k$ to $\Mix_k$. The heart remains exactly the same since the category $\NChow_k$ is by construction idempotent complete. By \cite[Defs.~1.1.1 and 1.2.1]{Bondarko} we then conclude that conditions (i)-(vi) of Theorem~\ref{thm:main} are verified, where $\Mix_k^{w\geq 0}$ (resp. $\Mix_k^{\leq 0}$) is the smallest idempotent complete and extension-stable subcategory of $\Mix_k$ (see \cite[Def.~1.3.1]{Bondarko}) containing the objects $\Phi(\NChow_k)[n], n \leq 0$ (resp. $\Phi(\NChow_k)[n], n \geq 0$). It remains then to verify condition (vii). We start by showing the equality $\cap_{l\in \bbZ} \Mix_k^{w \geq 0}[-l] =\{0\}\,.$ 
\begin{proposition}\label{prop:aux1}
For every noncommutative mixed motive $M$ there exists an integer $j \in \bbZ$ (which depends on $M$) such that for every $N \in \Phi(\NChow_k)$ we have 
\begin{eqnarray}\label{eq:equality-key}
\Hom_{\Mix_k}(N,M[i])=0 & \text{when} & i >j\,.
\end{eqnarray}
\end{proposition}
\begin{proof}
Let $\cC$ be a full subcategory of $\Madd(e)$ containing the zero object. Let us denote by $\cC[\bbZ]$ the category $\cup_{n \in \bbZ}
 \cC[n]$, by $\cC^\natural$ the idempotent completion of $\cC$ inside $\Madd(e)$, and by $\mathrm{Ext}(\cC)$ the subcategory of $\Madd(e)$ formed by the objects $\cO$ for which there exists a distinguished triangle
 \begin{equation}\label{eq:triangle}
 M_1 \too \cO \too M_2 \too M_1[1]
 \end{equation}
 with $M_1$ and $M_2$ in $\cC$. Note that $\cC \subseteq \mathrm{Ext}(\cC)$. Consider the following
 
 \medbreak

{\it Vanishing Condition:} there exists an integer $j \in \bbZ$ such that for every object $N \in \Phi(\NChow_k)$ we have
\begin{eqnarray*}
\Hom_{\Madd(e)}(N, \cO[i])=0 & \text{when} & i >j\,.
\end{eqnarray*}
\medbreak 

We now show that if by hypothesis the above vanishing condition holds for every object $\cO$ of $\cC$, then it holds also for every object of the following categories:

\begin{itemize}
\item[(1)] The category $\cC[\bbZ]$: this is clear since every object in $\cC[\bbZ]$ is of the form $\cO[n]$, with $n$ and integer and $\cO \in \cC$;
\item[(2)] The category $\mathrm{Ext}(\cC)$: by construction every object $\cO$ of $\mathrm{Ext}(\cC)$ fits in the above distinguished triangle \eqref{eq:triangle}. Let $j_1$ and $j_2$ be the integers of the vanishing condition which are associated to $M_1$ and $M_2$, respectively. Then, by choosing $j:=\mathrm{max}\{j_1, j_2\}$ we observe that the object $\cO$ also verifies the above vanishing condition;
\item[(3)] The category $\cC^\natural$: this is clear since every object in $\cC^\natural$ is a direct summand of an object in $\cC$; recall that $\Madd(e)$ admits arbitrary sums and so every idempotent splits.
\end{itemize}
Let us now apply the above general arguments to the category $\cC=\Phi(\NChow_k)$. By computation \eqref{eq:computation} the above vanishing condition holds for every object (with $j=0$). Recall that $\Mix_k$ is the smallest thick triangulated subcategory of $\Madd(e)$ spanned by the objects $N \in \Phi(\NChow_k)$. Hence, every object $M \in \Mix_k$ belongs to the category obtained from $\Phi(\NChow_k)$ by applying the above constructions (1)-(3) a {\em finite} number of times (the number of times depends on $M$). As a consequence, we conclude that $M$ satisfies the above vanishing condition and so the proof is finished.
\end{proof}
Let $M \in \cap_{l \in \bbZ} \Mix_k^{w \geq 0}[-l]$. Note that equality \eqref{eq:equality-key} can be re-written as
\begin{eqnarray}\label{eq:equality-key1}
\Hom_{\Mix_k}(N[-i],M)=0 & \text{when} & i >j\,.
\end{eqnarray} 
Since $\Mix_k^{w \geq 0}$ is the smallest idempotent complete and extension stable subcategory of $\Mix_k$ containing the objects $\Phi(\NChow_k)[n], n \leq 0$, we conclude from \eqref{eq:equality-key1} that $\Hom_{\Mix_k}(\cO,M)=0$ for every object $\cO$ belonging to $\Mix_k^{w \geq 0}[-l]$ with $l >j$. Since by hypothesis $M \in \Mix_k^{w \geq 0}[-l]$ we then conclude by the Yoneda lemma that $M=0$ in $\Mix_k^{w \geq 0}[-l]$ (with $l >j$) and hence in $\cap_{l \in \bbZ}\Mix_k^{w \geq 0}[-l]$. Let us now prove the equality $\cap_{l \in \bbZ} \Mix_k^{w\leq 0} [-l] =\{0\}$.
\begin{proposition}\label{prop:aux2}
For every non-trivial noncommutative mixed motive $M$ there exists an integer $j \in \bbZ$, an object $N \in \Phi(\NChow_k)$, and a non-trivial morphism $f : N[j] \to M$.
\end{proposition}
\begin{proof}
We prove the following equivalent statement: if $\Hom_{\Mix_k}(N[n],M)=0$ for every integer $n \in \bbZ$ and object $N \in \Phi(\NChow_k)$, then $M=0$. Recall that $\Mix_k$ is the smallest thick triangulated subcategory of $\Madd(e)$ spanned by the objects $N \in \Phi(\NChow_k)$. The class of objects $\cO$ in $\Madd(e)$ satisfying the equalities
\begin{eqnarray*}
\Hom_{\Madd(e)}(\cO[n],M)=0 && n \in \bbZ
\end{eqnarray*}
is clearly stable under extensions and direct factors. Since by hypothesis it contains the objects $N \in \Phi(\NChow_k)$ it contains also all the objects of the category $\Mix_k$. Hence, by taking $\cO=M$ and $n=0$, the identity morphism of $M$ allows us to conclude that $M=0$. 
\end{proof}
Let $M \in \cap_{l \in \bbZ} \Mix_k^{w \leq 0}[-l]$. If by hypothesis $M$ is non-trivial, then the morphism $f$ of Proposition~\ref{prop:aux2} gives rise to to a non-trivial morphism
\begin{equation}\label{eq:non-trivial}
0 \neq f[-j]:N \too M[-j]\,.
\end{equation}
Since by construction $N$ belongs to $\Mix_k^{w \geq 0}$, condition (iii) of Theorem~\ref{thm:main} combined with the non-trivial morphism \eqref{eq:non-trivial} implies that $M[-j] \notin \Mix_k^{w \leq 0}[1]$. Hence, $M \notin \Mix_k^{w \leq 0}[1+j]$ and so we obtain a contradiction with our hypothesis. This allows us to conclude that $M=0$ and so the proof of Theorem~\ref{thm:main} is finished.
\subsection*{Proof of Proposition~\ref{prop:aux}}
The category $\dgcat_k$ carries a (cofibrantly generated) Quillen model structure whose weak equivalences are precisely the derived Morita equivalences; see \cite[Thm.~5.3]{IMRN}. Hence, it gives rise to a well-defined Grothendieck derivator $\HO(\dgcat_k)$; consult \cite[Appendix~A]{CT1} for the notion of Grothendieck derivator. Since by hypothesis $\cM$ is stable and $L(-)$ satisfies conditions (i)-(iii) of Definition~\ref{def:additive}, we obtain then a well-defined additive invariant of dg categories $\HO(\dgcat_k) \to \HO(\cM)$ in the sense of \cite[Notation~15.5]{Duke}. By the universal property of \cite[Thm.~15.4]{Duke} this additive invariant factors through $\Madd$ giving rise to a homotopy colimit preserving morphism of derivators $\Madd \to \HO(\cM)$ and hence to a triangulated functor $\Madd(e) \to \Ho(\cM)$ on the underlying base categories. As explained in the proof of Theorem~\ref{thm:main}, the category $\Mix_k$ can be identified with a full triangulated subcategory of $\Madd(e)$. The composition obtained
$$ L(-): \Mix_k \subset \Madd(e) \too \Ho(\cM)$$
is then the triangulated functor mentioned in Proposition~\ref{prop:aux}. 

\subsection*{Proof of Theorem~\ref{thm:main2}}
As explained in the proof of Theorem~\ref{thm:main}, the category $\Mix_k$ is endowed with a non-degenerate bounded weight structure $w$ whose heart is equivalent to the category $\NChow_k$ of noncommutative Chow motives. In particular we have the following equalities
$$ \Mix_k^+ =\Mix_k = \Mix_k^{-}\,;$$
see \cite[Def.~1.3.5]{Bondarko}. Hence, item (i) follows from the combination of \cite[Thm.~3.2.2 II]{Bondarko} with \cite[Thm.~3.3.1]{Bondarko}. By Proposition~\ref{prop:aux} every additive invariant $L(-)$ gives rise to a triangulated functor $L(-): \Mix_k \to \Ho(\cM)$ and hence to a composed functor
\begin{equation}\label{eq:key-functor}
\xymatrix{
 \Mix_k \ar[r]^-{L(-)} & \Ho(\cM) \ar[rr]^-{\Hom({\bf 1},-)}&& \mathsf{Ab}\,.
}
\end{equation}
Note that \eqref{eq:key-functor} is {\em homological}, \ie it sends distinguished triangles to long exact sequences, and that we have the following identifications:
\begin{equation}\label{eq:identification}
\Hom({\bf 1}, L(M[-i])) \simeq \Hom({\bf 1}, L(M)[-i]) \simeq \Hom({\bf 1}[i],L(M)) =L_i(M) \,.
\end{equation}
Since the weight structure $w$ is bounded, we have $\Mix_k = \Mix_k^b$; see \cite[Def.~1.3.5]{Bondarko}. Hence, item (ii) follows from \cite[Thm.~2.3.2 II and IV]{Bondarko} (with $H=\eqref{eq:key-functor}$) and from the above identifications \eqref{eq:identification}. Finally, item (iii) follows from \cite[Thm.~2.3.2 III]{Bondarko} since all the Chern characters \eqref{eq:Chern} are natural transformations of additive invariants.

\begin{remark}
As the above proof clearly shows, Theorem~\ref{thm:main2}(ii) applies also to periodic cyclic homology; see Example~\ref{ex:HP}.
\end{remark}
\subsection*{Proof of Theorem~\ref{thm:main3}}
As explained in \cite[Thm.~7.5]{CT1} the functor \eqref{eq:functor} is symmetric monoidal. Since \eqref{eq:nat-functor} is also symmetric monoidal we conclude from the construction of \eqref{eq:fully-faithful} that this latter functor is also symmetric monoidal. Recall from the proof of Theorem~\ref{thm:main} that the category $\Mix_k$ is endowed with a bounded weight structure $w$. Since the functor \eqref{eq:nat-functor} is symmetric monoidal the heart $\Phi(\NChow_k)$ of $w$ is then a full additive symmetric monoidal subcategory of $\Mix_k$. Hence, the result follows from the combination of \cite[Thm.~5.3.1]{Bondarko} with \cite[Remark~5.3.2]{Bondarko}.

\end{document}